\documentclass{article}
\usepackage{t1enc}
\usepackage{latexsym}
\usepackage{amssymb}
\usepackage{amsmath}
\usepackage{amsthm}
\usepackage[mathscr]{eucal}
\usepackage{graphicx}
\usepackage{pictexwd}

\newcommand{\bibname}{}

\newtheorem{thm}{Theorem}
\newtheorem{lem}{Lemma}

\newtheorem*{lem*}{Lemma}

\theoremstyle{definition}
\newtheorem{dfn}{Definition}
\newtheorem{rem}{Remark}

\newcommand{\Cob}{\mathscr{N}}

\newcommand{\Sr}[1]{\Sigma^{1_{#1}}}

\newcommand{\ka}[3]{\hat{#1}_{#2}(#3)}
\newcommand{\huN}[2]{\widetilde{\Delta_{#1}}(#2)}
\newcommand{\huM}[2]{\Delta_{#1}(#2)}
\newcommand{\gN}[2][f]{\widetilde{#1_{#2}}}
\newcommand{\Lr}[1][i]{\Lambda_r^{#1}}
\newcommand{\lr}[1][i]{\lambda_r^{#1}}

\newcommand{\R}{\mathbb{R}}

\newcommand{\numeq}[1]{\stackrel{\mbox{\scriptsize (#1)}}{=}}

\newcommand{\immto}{\looparrowright}
\newcommand{\embto}{\hookrightarrow}

\newcommand{\un}[1]{\underline{#1}}
\newcommand{\uun}[1]{\un{\un{#1}}}
\newcommand{\ie}{i.e{.}}
\newcommand{\eg}{e.g{.}}

\title{Singularities of projected immersions revisited\thanks{This research was partially
supported by grant T049449 of the Hungarian Scientific Research Fund. MSC2000:
57R45, 57R42}}
\author{G\'abor Lippner}

\begin{document}
\maketitle

\begin{abstract}
Sz\H ucs proved in~\cite{Sz1} that the $r$-tuple-point manifold of a generic
immersion is cobordant to the $\Sr{r-1}$-point manifold of its generic
projection. Here we slightly extend this by showing that the natural mappings
of these manifolds are bordant to each other. The main novelty of our approach
is that we construct the bordism explicitly.

\end{abstract}

\section{Introduction}
There is a surprising relation between the multiple-points of an immersion $g :
M \immto N \times \R$ and the singularities of its projection $f: M \to N$ that
was found by Sz\H ucs in~\cite{Sz1} (see also~\cite{Sz2}).  Namely he showed
that if $N$ is a Euclidean space then the $r+1$-tuple-points of $g$ are
cobordant to the $\Sr{r}$ points of $f$. The proof of this result involved
computing the characteristic numbers of the two manifolds and observing that
they coincide.

It is very natural to ask whether this cobordism can be ``seen'' in an
explicit way hidden in the geometry of $f$, not just as mere luck that
all the characteristic numbers coincide.

We shall answer this question in the affirmative by constructing a
cobordism that connects the two manifolds. This allows us to slightly
extend the original theorem: instead of cobordism of manifolds we
obtain singular bordism of maps, and we prove the theorem for any
smooth target manifold $N$. (The notation and the necessary
definitions are given in the next section.)

\begin{thm}
Let $f: M^n \to N^{n+k}$ be a prim map, and let $g: M \immto N \times \R$ be
its lift to an immersion. Then for any $r \geq 1$ we have $g_{r} \sim
(\Sr{r-1}(f) \embto M)$, that is they represent the same element in the singular bordism
group $\Cob_*(M)$.

If $M$ and $N$ are oriented and the codimension $k$ is odd, then
$g_{r} \sim_{SO} (\Sr{r-1}(f) \embto M)$, that is they represent the
same element in the singular oriented bordism group $\Omega(M)$.
\end{thm}

\section{Definitions and notation}

Consider a proper, generic immersion $g : M^n \to N^{n+k}$ of a closed
smooth manifold $M$ to a smooth manifold $N$. The \textit{$r$-fold
points} of $g$ are those points in $N$ whose preimage consists exactly
of $r$ different points. We shall denote this set $N_r$. This is not
always a closed set in $N$. Its closure $\bar{N}_r$ consists of those
points that have at least $r$ distinct preimages. Put $M_r =
g^{-1}(N_r)$, this is the set of $r$-tuple points of $g$ in the source
manifold. Its closure is denoted by $\bar{M}_r$.

The sets $M_r$ and $N_r$ are generally not submanifolds of $M$ and $N$
but they are images of (non-generic) immersions of (possibly open)
manifolds. Here we recall a well-known construction (see
\eg~\cite{Ro}) to fix the notation: Let \[ \ka{M}{r}{g} =
\{(x_1,\dots,x_r) \in M^{(r)} : g(x_1)=\cdots=g(x_r), (i \neq j)
\Rightarrow (x_i \neq x_j)\}.\]

The symmetric group $S_r$ acts on this set freely in an obvious
way. Let $[x_1,\dots,x_r]$ denote the equivalence class of
$(x_1,\dots,x_r)$. On the other hand $S_{r-1}$ also acts freely on the
last $r-1$ coordinates. Here the equivalence class of
$(x_1,\dots,x_r)$ is denoted by $(x_1,[x_2,\dots,x_r])$.

\begin{dfn}
The sets of equivalence classes are denoted by
\begin{align*} \huN{r}{g} &= \ka{M}{r}{g}/S_r \\
               \huM{r}{g} &= \ka{M}{r}{g}/S_{r-1}.
\end{align*}
There are obvious mappings
\begin{align*} \gN[g]{r} : \huN{r}{g} &\to N & \gN[g]{r}([x_1,\dots,x_r]) &:=
  g(x_1)\\ g_r : \huM{r}{g} &\to M & g_r(x_1,[x_2,\dots,x_r]) &:=
  x_1\\ s_r : \huM{r}{g} &\to \huN{r}{g} & s_r(x_1,[x_2,\dots,x_r])
  &:= [x_1,\dots,x_r].
\end{align*}
\end{dfn}

 The images of $\gN[g]{r}$ and $g_r$ are clearly $\bar{N}_r$ and $\bar{M}_r$ and
they are bijective to the points that have multiplicity exactly
$r$. On the other hand $s_r$ is clearly an $r$-sheeted covering.

The sets $\huN{r}{g}$ and $\huM{r}{g}$ are called the $r$-fold
multiple-point manifolds of $g$ in the target and source
respectively. They are indeed manifolds.  To see this we need the
notion of the fat and narrow diagonals. Let $V$ be a manifold and
$V^{(r)}$ its $r$-fold Cartesian product. Then let $\delta_r(V) =
\{(x,x,\dots,x) \in V^{(r)} | x \in V\}$ and $\Delta_r(V) =
\{(x_1,x_2,\dots,x_r) \in V^{(r)} | \exists i \neq j, x_i = x_j \}$
denote the narrow and the fat diagonals respectively.  Consider the
$r$-fold product $g^{(r)}: M^{(r)} \to N^{(r)}$.  Clearly
\[\ka{M}{r}{g} = (g^{(r)})^{-1}(\delta_r(N))
\setminus \Delta_r(M).\] Since $g$ is a generic immersion, $g^{(r)}$
is transverse to $\delta_r(N)$ and thus $\ka{M}{r}{g}$ is a closed
manifold of dimension $n-(r-1)k$. The symmetric group $S_r$ acts on it
freely, so after factoring out with the group actions of $S_r$ or
$S_{r-1}$ we still get manifolds.

\begin{rem} If $M$ and $N$ are oriented and the codimension is even, then the
multiple-point manifolds can be given a natural orientation. If $k$ is odd
however, then the action of $S_r$ contains orientation reversing elements, thus
the factor manifolds will have no, or at least no preferred orientation.
\end{rem}

\begin{dfn}
Given a smooth map $f : M \to N$ where $\dim M \leq \dim N$, a point
$x \in M$ is said to be a $\Sigma^{i}$ point if the corank (\ie the
dimension of the kernel) of $df_x : T_xM \to T_{f(x)}N$ is $i$. The
closure of the set of such points will be denoted by
$\Sigma^{i}(f)$. If $i_1 \geq i_2$ then we can define
$\Sigma^{i_1,i_2}(f) = \Sigma^{i_2}(f|_{\Sigma^{i_1}(f)})$. This
method can be continued recursively to give the definition of
$\Sigma^{(i_1,i_2,\dots,i_r)}$ points, where $i_1 \geq i_2 \geq \dots
\geq i_r$. This classification of singular points is called the
Thom-Boardman type. For details see e{.}  g{.}~\cite{AGLV}.
\end{dfn}

\begin{dfn}~\label{sr}
A map $f: M \to N$ is said to be a \textit{Morin} map if it has no
$\Sigma^2$ points. The singularities of such maps are classified by
their Thom-Boardman type, which can only be
$\Sigma^{\overbrace{(1,1,\dots,1)}^r} = \Sr{r}$ for some $r \geq
0$. (In the notation of~\cite{AGLV} this is $A_r$.) The set
$\Sr{r}(f)$ is actually a submanifold of $M$.
\end{dfn}

\begin{dfn}
A generic map $f: M \to N$ is called prim (\textit{pr}ojected
\textit{im}mersion) if it has a specified lifting to a generic immersion, $g :
M \immto N \times \R$ (\ie $f=\pi \circ g$, where $\pi:N\times \R \to
N$ is the projection). This lifting $g$ has to be given up to regular homotopy.
Such a map is necessarily a Morin map (i{.}e{.} its differential has corank at
most one at any point), and so its singularities are classified by their
Thom-Boardman type.
\end{dfn}

\begin{dfn}
The fat diagonal of $\Sr{r}(f) \times M^{(i-1)}$ can be defined analogously to
$\Delta_i(M)$, since $\Sr{r}(f) \subset M$ is a submanifold. Let us denote
\[ \Delta_i^r(M) = \{ (x_1,x_2,\dots,x_i) \in \Sr{r}(f)
\times M^{(i-1)} : \exists j\neq l, x_j=x_l\}.\]
\end{dfn}

\begin{rem}
For any manifold $M$ we shall denote its cobordism class by $[M] \in
\Cob_*$ and for a map $f: M \to N$ we shall denote its singular bordism
class by $[f] \in \Cob_*(N)$. The cobordism relation for both manifolds
and maps will be denoted by a $\sim$. If $M$ is oriented then the same
notation will be used for the corresponding classes in $\Omega_*$ and
$\Omega_*(N)$ respectively.
\end{rem}

\section{Proof of the theorem}

\subsection{Preparations}~\label{prepsec}

Let us fix a prim map $f: M^n \to N^{n+k}$, its lift $g: M \immto N
\times \R$ and an integer $r \geq 2$ (for $r=1$ the statement is
obvious). We shall introduce auxiliary manifolds and their maps to $M$
which we shall call 'mixed'-point manifolds. For any $1 \leq i \leq r$
let us consider those points in $M$ that are $i$-tuple points of $g$
and at the same time $\Sr{r-i}$ points of $f$.  These points do not
necessarily form a submanifold of $M$, but we can construct their
resolution just like we did for the set of $r$-tuple points of an
immersion: Let us consider the map
\[ G_i := g|_{\Sr{r-i}(f)} \times g \times \dots \times g : \Sr{r-i}(f) \times M
\times \dots \times M \to (N \times \R) \times \dots \times (N \times
\R),\] where we take $i-1$ factors of $M$ on the left, and thus $i$
factors of $(N\times \R)$ on the right. Since $f$ is a generic prim
map and $g$ is its generic lift we have that $G_i$ is transverse to
the narrow diagonal $\delta_i(N\times \R)$ outside of the fat diagonal
$\Delta_i^{r-i}(M)$. Since $g$ is an immersion the set
$\ka{M}{i,r-i}{f} := G_i^{-1}(\delta_i(N\times \R)) \setminus
\Delta_i^{r-i}(M)$ is a closed submanifold in $\Sr{r-i}(f) \times
M^{(i-1)}$. The symmetric group $S_{i-1}$ acts on $\Sr{r-i}(f) \times
M^{(i-1)}$ by permuting the last $i-1$ coordinates. This action
restricted to $\ka{M}{i,r-i}{f}$ is free, so we can factorize and get
the manifold \[\Lr = \ka{M}{i,r-i}{f}/S_{i-1}.\] A point of $\Lr$ can
be referred to as $(x_1,[x_2,\dots,x_i])$ where the $x_j$'s are all
different, $g(x_1) = g(x_2) = \dots = g(x_i)$ and $x_1 \in
\Sr{r-i}(f)$. In this notation the desired resolution \[\lr : \Lr \to
M\] is given by \[ (x_1,[x_2,\dots,x_i]) \mapsto x_1.\] (The maps
$f,g$ are omitted from the notation.) It is easy to see that the
manifold $\Lr$ has dimension \mbox{$n-(r-1)(k+1)$} and in particular
$\Lr[r] = \huM{r}{g}, \lr[r] = g_r$ and $\lr[1] :\Lr[1] \to M$ is the
natural inclusion $\Sr{r-1}(f) \embto M$. Thus the theorem follows
from the following lemma.

\begin{lem}~\label{mainlemma}
$\lr[1] \sim \lr[2] \sim \dots \sim \lr[r]$, \ie~hese maps represent
  the same class in $\Cob_*(N)$.
\end{lem}

The proof consists of two very different ingredients. The first
ingredient is the global construction of the desired cobordisms using
the map $f$. The constructed spaces are easy to describe but they are
not obviously manifolds. The precise proof that they are indeed
manifolds requires detailed study of the map $f$ near its singular
points. Thus the second ingredient is a local computation using normal
forms. This computation is only a technical point so first we give the
proofs omitting the computational details. Then in
section~\ref{computations} we finally show how to carry out the
computations used earlier.

\subsection{Proof of Lemma~\ref{mainlemma}}

Let us again consider the map
\[ G_i := g|_{\Sr{r-i}(f)} \times g \times \dots \times g :
\Sr{r-i}(f) \times M \times \dots \times M \to (N \times \R) \times
\dots \times (N \times \R).\] Let us define
\[ \Delta_i^+ = \{ ((x,s),(x,t),\dots,(x,t)) \in (N \times \R)^{(i)}: s \geq
t\}.\]

Outside of $\Delta_i^{r-i}(M)$ the map $G_i$ is transverse to $\Delta_i^+$ and
$\partial \Delta_i^+ = \delta_i(N\times \R)$, since both $f$ and $g$ are
generic and thus self-transverse.

Let us now define $H' = G_i^{-1}(\Delta_i^+) \setminus \Delta_i^{r-i}(M)$.
Transversality implies that $H'$ is a (not necessarily compact) manifold with
boundary $G_i^{-1}(\delta_i(N\times \R))\setminus \Delta_i^{r-i}(M) =
\ka{M}{i,r-i}{f}$. Let us denote the closure of $H'$ in $\Sr{r-i}(f)\times
M^{(i-1)}$ by $H$. Obviously $H \setminus H' \subset \Delta_i^{r-i}(M)$. We
have seen in section~\ref{prepsec} that $\partial H'$ is a closed manifold
disjoint from the fat diagonal. Thus $\partial H'$ is disjoint from $H
\setminus H' \subset \Delta_i^{r-i}(M)$.

Let us take a point $(x_1,\dots,x_i) \in H\setminus H'$. Then by definition of
$H'$ there exist points $y_j^k \ (k\geq 1, i\geq j\leq 1)$ that fulfill all the
following requirements:
\begin{enumerate}
\item For every $j$ we have $\lim_{k \to \infty}y_j^k = x_j$.
\item $y_1^k \in \Sr{r-i}(f)$.
\item For any fixed $k$ the $y_j^k$'s are all different.
\item $g(y_{j_1}^k) = g(y_{j_2}^k)$ for any $j_1,j_2 \geq 2$.
\item $f(y_{j_1}^k) = f(y_{j_2}^k)$ for any $j_1,j_2 \geq 1$.
\end{enumerate}

Since $g$ is a generic immersion, 3. and 4. imply that $\forall
j>l\geq 2$ we have $x_j \neq x_l$. Then since $H \setminus H' \subset
\Delta_i^{r-i}(M)$ there must be a $j>1$ such that $x_1 = x_j$.  Thus
$y_1^k \to x_1$ and $y_j^k\to x_1$ as well. Furthermore $y_1^k \in
\Sr{r-i}(f)$. Theorem~\ref{lokalisszamolas} in
section~\ref{computations} can be applied and hence $x_1 \in
\Sr{r-i+1}(f)$.

Conversely let us suppose that $x_1 \in \Sr{r-i+1}(f)$ and
$x_2,\dots,x_{i-1}$ are all different from each other and $x_1$ and
$g(x_j)$ is the same for every $1\leq j \leq i-1$. We want to show
that in the neighborhood of $(x_1,x_1,x_2,\dots,x_{i-1})$ the set $H$
is a compact manifold with boundary and $(x_1,x_1,x_2,\dots,x_{i-1})$
is on $\partial H$. First consider the first two factors separately
from the others.
\[ G_2 = g|_{\Sr{r-i}(f)} \times g : \Sr{r-i}(f) \times
M \to (N \times \R)^{(2)}.\] Let us denote $H_2' = G_2^{-1}(\Delta_2^+)$. By
Theorem~\ref{altalanosszamolas} in section~\ref{computations} we know that
locally around $(x_1,x_1)$ its closure $H_2 = \mbox{cl}(H_2')$ is a compact
manifold with boundary $\partial H_2 = \{(u,u) : u \in \Sr{r-i+1}(f)\}$.
Clearly $H$ is locally the complete intersection of $H_2 \subset \Sr{r-i}(f)
\times M$ around $(x_1,x_1)$ and $\ka{M}{i-2}{g} \subset M^{(i-2)}$ around
$(x_2,\dots,x_{i-1})$. Thus the genericity of $f$ and $g$ implies that $H$ is
also locally a compact manifold with boundary $\partial H$ the complete
intersection of $\partial H_2$ and $\ka{M}{i-2}{g}$.

Thus $H$ is a compact manifold. Its boundary consist of two disjoint components
$H \setminus H'$ and $\partial H' = \ka{M}{i,r-i}{f}$. The symmetric group
$S_{i-1}$ acts on $\Sr{r-i}(f)\times M^{(i-1)}$ by permuting the last $i-1$
coordinates. By definition $H'$ is invariant under this action. The above
considerations show that $\partial H'$ and $H\setminus H'$ are also invariant,
and the action is free on each. Thus we can factorize by this action on $H$ and
get that the quotient is again a compact manifold $\hat{H}$ with boundary
$\partial H'/S_{i-1}$ and $(H\setminus H')/S_{i-1}$. By definition $\partial
H'/S_{i-1} = \ka{M}{i,r-i}{f}/S_{i-1} = \Lr[i]$. On the other hand we have seen
that
\[H \setminus H' = \{ (x_1,x_2,\dots,x_i) \in
\Delta_i^{r-i+1}(M)\setminus \Sr{r-i+1}(f)\times \Delta_{i-1}(M) :
g(x_j)=g(x_l)\ (1\leq j<l \leq i)\}.\]
Thus there is a natural map $\phi : (H\setminus H')/S_{i-1} \to \Lr[i+1]$
that is given by $\phi(x_1,[x_2,\dots,x_i]) =
(x_1,[x_2,\dots,x_{j-1},x_{j+1},\dots,x_i])$ when $x_1 = x_j$. This
map is clearly a diffeomorphism. Thus $(H\setminus H')/S_{i-1} =
\Lr[i+1]$.

Finally projecting everything to the first coordinate we get a map
$\hat{H} \to M$ that on the boundary coincides with $\lr[i]$ and
$\lr[i+1]$. Thus $\lr[i] \sim \lr[i+1]$. \qed

\begin{rem} If the codimension $k$ is odd, then the codimension of $g$ is even.
So if $M$ and $N$ are oriented, then $H'$ can be given a natural orientation.
This is preserved by the action of $S_{i-1}$ and so the manifold $\hat{H}$ that
creates the cobordism between $\lr[i]$ and $\lr[i+1]$ is oriented. Thus $\lr[i]
\sim_{SO} \lr[i+1]$ and the oriented part of theorem follows as well.
\end{rem}

\section{Local computations}~\label{computations}
Let us consider a prim map $f : M^n \to N^{n+k}$. Let us write $n =
r(k+1) + z$. Then the $\Sr{r}$-points of $f$ form a $z$-dimensional
submanifold in $M$. Let $x \in \Sr{r}(f) \setminus \Sr{r+1}(f)$. Then
(according to e{.} g{.}~\cite{AGLV}) it is possible to take small
Euclidean neighborhoods of $x$ and $f(x)$ and introduce local
coordinates such that $f$ takes the following local normal form (we
take both $x$ and $f(x)$ to be in the origin):
\begin{eqnarray*}
F : (\R^{r(k+1) +z},0) &\to& (\R^{1+k+(r(k+1)-1)+z},0)\\
 (\un{y}^r,\un{y}^{r-1},\dots,\un{y}^1,\un{s}) &\mapsto&
(p_0(y_0^k),p_1(y_0^k),\dots,p_k(y_0^k),\uun{y},\un{s}),
\end{eqnarray*} where
$\un{y}^j = (y_0^j,y_1^j,\dots,y_k^j) \in \R^{k+1}$ for every $1 \leq j \leq k$
and $p_i(x) = \sum_{j=1}^r y^j_i x^j$ are polynomials. By $\uun{y}$ we denote
the collection of all $y_i^j$ except $y_0^k$, so $\uun{y} \in \R^{r(k+1)-1}$.
Finally $\un{s} = (s_1,\dots,s_z) \in \R^z$. The variable $y_0^k$ is special
and will also be denoted by $t$. Note that $p_0(y_0^k) = p_0(t) = t^{r+1} +
y_0^{r-1}t^{r-1} + \dots + y_0^1 t$ is actually degree $r+1$ in $t$, while for
any $i > 0$ we have $p_i(y_0^k) = p_i(t) = y_i^r t^r + \dots y_i^1 t$ which is
degree $r$ in $t$. We will think of the $p_i$ mostly as polynomials of the
single variable $t$.

\begin{lem}~\label{szigmapontok}
\begin{enumerate}
\item
The point $(t,\uun{y},\un{s})$ is a $\Sr{j}$-point of $F$ if and only
if $p_i'(t) = p_i''(t) = \dots = p_i^{(j)}(t) = 0$ for every $0 \leq i
\leq k$.
\item The set of such
points form a submanifold in $\R^{r(k+1)+z}$ which can be smoothly
parametrized by $\un{s},\un{y}^r,\dots,\un{y}^{j+1}$.
\end{enumerate}
\end{lem}

\begin{proof}
Part 2 easily follows from part 1, since if $j<r$ and
$\un{s},\un{y}^r,\dots,\un{y}^{j+1}$ are fixed, then $p_i^{(j)}(t) =
0$ is a non-degenerate linear equation for $\un{y}^j$. This can be
uniquely solved. Then $p_i^{(j-1)}(t) = 0$ is a non-degenerate linear
equation for $\un{y}^{j-1}$, and so on. Finally if $j=r$ then
obviously the only solution is $y_i^l = 0$ for every $i,l$
independently of $\un{s}$. Thus it suffices to show part 1.

We will proceed by induction on $j$. The initial step $j=1$ is easy to see:
$dF$ is singular if and only if $p_i'(t) = 0$ for every $i$ and in this case
$\ker dF$ is the $t$-axis. Now let us suppose we know the statement for $j-1$
and take a point $x \in \Sr{j}(F)$. Then $x \in \Sr{j-1}(F)$ and $\ker d_x F
\subset T_x\Sr{j-1}(F)$. Then there is a sequence of points $x(i) =
(t(i),\uun{y}(i),\un{s}(i) \in \Sr{j-1}(F)$ such that $x(i) \to x =
(t,\uun{y},\un{s})$, $\frac{t(i)-t}{|x(i)-x|} \to 1$ and
$\frac{|\uun{y}(i)-\uun{y}|}{|x(i)-x|} \to 0$. Let us focus on $p_l$ where $l$
is arbitrary but fixed, and temporarily denote it by $p$. We will also
temporarily include in the notation of $p$ all its hidden variables. Then
\[ p^{(j)}(t,\uun{y}) = \lim_{i \to \infty}
\frac{p^{(j-1)}(t(i),\uun{y})-p^{(j-1)}(t,\uun{y})}{t(i)-t} \numeq{1}
\lim_{i \to \infty}
\frac{p^{(j-1)}(t(i),\uun{y})-p^{(j-1)}(t(i),\uun{y}(i))}{t(i)-t}
\numeq{2} 0.\] Here (1) holds since
$p^{(j-1)}(t,\uun{y})=p^{(j-1)}(t(i),\uun{y}(i))=0$ by the inductive
hypothesis. (2) holds since $p^{(j-1)}$ is a fixed finite sum of
expressions linear in $\uun{y}$ and
$\frac{|\uun{y}(i)-\uun{y}(i)|}{t(i)-t} \to 0$. This argument can be
easily reversed and so the statement is true for $j$. This completes
the induction and thus proves part 1.
\end{proof}

\begin{thm}~\label{lokalisszamolas}
Let $f: M^n \to N^{n+k}$ a generic Morin map. If there exist points
$x_i\neq x_i' \in M; (i\geq 1)$ such that $x_i \to x$, $x_i'\to x$,
$x_i \in \Sr{r}(f)$ and $f(x_i) = f(x_i')$ for every $i$, then $x \in
\Sr{r+1}(f)$.
\end{thm}

\begin{proof}
It is obvious that $x \in \Sr{r}(f)$. Let us suppose that $x \in
\Sr{r}(f) \setminus \Sr{r+1}(f)$. We can consider $f$ locally around
$x$ and introduce Euclidean neighborhoods as before, denoting the
function in the new coordinate-system by $F$. As $x_i \to x$ and $x_i'
\to x$, these points will fall into the chosen neighborhood with at
most finite exceptions. From Lemma~\ref{szigmapontok} it is obvious
that the only $\Sr{r}$-points of $F$ are those for which $t=0$ and
$\uun{y} = \uun{0}$, and $\un{s}$ is arbitrary. On the other hand if
$F(t,\uun{y},\un{s}) = (0,0,\dots,0,\uun{0},\un{s})$ then obviously
$t=0$ and $\uun{y} = \uun{0}$. So none of the $\Sr{r}$-points of $F$
are double points of $F$ which is contradiction.
\end{proof}

If $f: M^n \to N^{n+k}$ is actually a prim map with lifting $g: M^n
\immto N^{n+k}\times \R$ and $x\in \Sr{r}(f)\setminus \Sr{r+1}(f)$,
then we can take the Euclidean coordinates around $x$ and $f(x)$
introduced at the beginning of this section, and choose a last extra
coordinate around $g(x)$ such that $g$ takes the local form $G(x) =
(F(x),t)$.  Let $j<r$ and let us consider the set
\[ A' = \{ (u,v) \in \R^n \times \R^n : u \in \Sr{j}(F), F(u) = F(v),
t(u)\geq t(v)\}\] and its closure $A = \mbox{cl}(A')$.

\begin{thm}~\label{altalanosszamolas}
The set $A$ is a manifold with boundary $\partial A = \{(u,u): u\in
\Sr{j+1}(F)\}$.
\end{thm}

\begin{proof} Theorem~\ref{lokalisszamolas} implies that a boundary point
of $A'$ must be in $\Sr{j+1}(F)$. We shall give an explicit smooth
parametrization of $A'$ on an open halfspace, and show that this
extends smoothly and bijectively to a parametrization of $\Sr{j+1}(F)$
on the boundary of the halfspace. It is obvious that the variables
$\un{s}$ play no role whatsoever, so without loss of generality we may
assume that $z = 0$ and omit $\un{s}$ from the further calculations.

The condition $F(u) = F(v)$ obviously implies $\uun{y}(u) =
\uun{y}(v)$, so $(u,v) \in A'$ if and only if $t(u) > t(v)$, and
$p_i'(t(u)) = p_i''(t(u)) = \dots = p_i^{(j)}(t(u)) =
p_i(t(u))-p_i(t(v)) = 0$ for every $i$. (Here we think of $p_i$ as a
polynomial of one variable. Its coefficients depend on $\uun{y}$, but
since $\uun{y}$ is independent of $u$ and $v$, this notation makes
sense.)

We claim that for any choice of parameters $t(v) > t(u), \un{y}^r,
\un{y}^{r-1},\dots,\un{y}^{j+2}$ there is a unique choice of
$\un{y}^{j+1},\dots,\un{y}^1$ depending smoothly on the parameters
such that the resulting pair of points $(u,v) \in A'$. (In case of $j
= r-1$ there is only a single parameter $t(v)>0$.)

Let us first deal with the case $j < r-1$. Then for each $i$ the
problem of finding $y^{j+1}_i,y^j_i,\dots,y^1_i$ such that $p_i'(t(u))
= p_i''(t(u)) = \dots = p_i^{(j)}(t(u)) = p_i(t(u))-p_i(t(v)) = 0$
holds can be solved independently of each other. In fact the problem
is the same for every $i$, so we fix an arbitrary $i$ and denote
$p_i(t)=p(t) = \lambda_r t^r + \dots + \lambda_1 t$ temporarily.  Let
us write $p(t) = q(t) + \lambda_{j+2}t^{j+2}+\dots+\lambda_r t^r =
q(t) + r(t)$. Since $\lambda_r,\dots,\lambda_{j+2}, t(u)$ and $t(v)$
are fixed parameters, we know the value of
$r(t(u)),r(t(v)),r'(t(u)),r''(t(u)),\dots, r^{(j)}(t(u))$. We have to
find the coefficients of $q$. Let us write $q$ as a Taylor polynomial
around $t(u)$. Then
\begin{equation}~\label{taylor}
q(t) = q(t(u)) + \sum_{i=1}^j q^{(i)}(t(u))\cdot \frac{(t-t(u))^i}{i!}
+ \lambda_{j+1}\cdot(t-t(u))^{j+1} .
\end{equation}
Since $0 = p^{(i)}(t(u)) = q^{(i)}(t(u))+r^{(i)}(t(u))$, in
(\ref{taylor}) the only unknown value is $\lambda_{j+1}$. By
definition \[q(t(v))
- q(t(u)) = p(t(v)) - r(t(v)) + r(t(u))- p(t(u)) = r(t(u)) -
r(t(v)),\]
and hence by substituting $t = t(v)$ in (\ref{taylor}) we get that
\begin{equation*}
\lambda_{j+1} = \left(\frac{1}{t(v)-t(u)}\right)^{j+1} \cdot \left(
r(t(u))-r(t(v))- \sum_{i=1}^j q^{(i)}(t(u))\cdot
\frac{(t-t(u))^i}{i!}\right)
\end{equation*}
As every quantity on the right hand side is fixed and $t(u) > t(v)$ we
find that the parameters uniquely and smoothly determine
$\lambda_{j+1}$. Then all the remaining $\lambda$'s are uniquely and
smoothly determined by the Taylor expansion (\ref{taylor}). Finally to
see what happens on the boundary of the halfspace $t(u) > t(v)$ just
observe, that the vanishing of the derivatives of $p$ at $t(u)$ imply
that $p(t) = p(t(u)) + (t-t(u))^{j+1}\cdot w(t)$ for some polynomial
$w(t)$. Then the equation $p(t(v))=p(t(u))$ is equivalent to
$w(t(v))=0$. Then if $t(v)-t(u)$ converges to 0 the solution will
converge to a $w(t)$ for which $w(t(u))=0$, which is equivalent to
saying that $p^{(j+1)}(t(u))=0$. So the boundary of the halfspace
$t(u) >t(v)$ parametrizes those points $(u,u)$ for which
$p'(t(u))=p''(t(u))=\dots=p^{(j+1)}(t(u))=0$ which is equivalent to $u
\in \Sr{j+1}(F)$.

Now consider the case $j=r-1$. The only parameter is $t(v)$. Let us
suppose that we have a solution $u$ that satisfies all the
equations. Let $i \geq 1$. Then $p_i(t)$ is a degree $r$ polynomial
for which the first $r$ derivatives vanish at $t(u)$. Thus $p_i = c_i
\cdot (t-t(u))^r$. Further we know that $p_i(t(v)) = p_i(t(u)) = 0$
while $t(v) > t(u)$. This is only possible if $c_i = 0$. So all the
$p_i$'s must be identically 0, except for $p_0$. Let us temporarily
denote $p_0(t)=p(t) = t^{r+1} + \lambda_{r-1}t^{r-1} + \dots +
\lambda_1 t$. The constraints on the derivatives imply that
\[p(t) =
p(t(u))+p^{(r)}(t(u))\cdot \frac{(t-t(u))^r}{r!} + (t-t(u))^{r+1}.\]
The polynomial $p$ has no $x^r$ term by definition, so
$p^{(r)}(t(u))=r!(r+1)t(u)$, and so
\[ p(t) = p(t(u))+(t+r\cdot t(u))(t-t(u))^r.\]
Finally
\[
p(t(u)) = p(t(v)) = p(t(u))+(t(v)+r\cdot t(u))(t(v)-t(u))^r,\]
so $t(u) = -t(v)/r$, and $p(0)=0$ determines $p(t(u))$.
Thus indeed for any $t(v)>0$ there is a unique solution $u$, this
solution is smoothly parametrized by $t(v)$, and the boundary $t(v)=0$
goes to the only $\Sr{r}$-point, the origin.
\end{proof}

\end{document}